\documentclass[a4paper,12pt]{article}
\usepackage[english]{babel}

\usepackage{amsmath,amssymb,amsthm}
\usepackage[pdftex]{color}
\usepackage[bookmarks=true,hyperindex,pdftex,colorlinks, citecolor=MidnightBlue,linkcolor=MidnightBlue, urlcolor=MidnightBlue]{hyperref}
\usepackage[dvipsnames]{xcolor}
\usepackage{enumerate}
\usepackage{lineno}
\usepackage[edges]{forest}
\usepackage{mathtools}
\usepackage{graphicx}
\usepackage{subfig,bbm}

\usetikzlibrary{decorations.markings}
\tikzset{neg/.style={
		decoration={markings,
			mark= at position 0.5 with {
				\node[transform shape] (tempnode) {$\setminus$};
			}
		},
		postaction={decorate}
}}

\numberwithin{equation}{section}

\newcommand{\MAT}{\left[ \begin{array}}  
	\newcommand{\mat}{\end{array} \right]}

\newtheorem{theorem}{Theorem}[section]

\newtheorem{corollary}[theorem]{Corollary}
\newtheorem{lemma}[theorem]{Lemma}

\theoremstyle{definition}

\newtheorem*{definition*}{Definition}
\newtheorem*{proposition*}{Proposition}
\newtheorem*{theorem*}{Theorem}
\newtheorem*{corollary*}{Corollary}
\newtheorem*{example*}{Example}
\newtheorem*{problem*}{Problem}

\theoremstyle{remark}
\newtheorem{remark}[theorem]{Remark}

\newcommand*{\prob}[1]{\mathbb{P}\left\{ #1 \right\}}

\def \ex{\mathbb{E}}

\newcommand{\interior}[1]{%
	{\kern0pt#1}^{\mathrm{o}}%
}

\title{Fast Dimensionality Reduction from $\ell_2$ to $\ell_p$}

\author{Rafael Chiclana\thanks{Michigan State University, Department of Mathematics, \emph{chiclan1@msu.edu}.} \and Mark Iwen\thanks{Michigan State University, Department of Mathematics, and Department of Computational Mathematics, Science and Engineering (CMSE), \emph{iwenmark@msu.edu}.  Supported in part by NSF DMS 2106472.}}

\begin{document}

\maketitle

\begin{abstract}
    The Johnson-Lindenstrauss (JL) lemma is a fundamental result in dimensionality reduction, ensuring that any finite set $X \subseteq \mathbb{R}^d$ can be embedded into a lower-dimensional space $\mathbb{R}^k$ while approximately preserving all pairwise Euclidean distances. In recent years, embeddings that preserve Euclidean distances when measured via the $\ell_1$ norm in the target space have received increasing attention due to their relevance in applications such as nearest neighbor search in high dimensions. A recent breakthrough by Dirksen, Mendelson, and Stollenwerk established an optimal $\ell_2 \to \ell_1$ embedding with computational complexity $\mathcal{O}(d \log d)$. In this work, we generalize this direction and propose a simple linear embedding from \( \ell_2 \) to \( \ell_p \) for any \( p \in [1,2] \) based on a construction of Ailon and Liberty. Our method achieves a reduced runtime of \( \mathcal{O}(d \log k) \) when \( k \leq d^{1/4} \), improving upon prior runtime results when the target dimension is small. Additionally, we show that for \emph{any norm} $\|\cdot\|$ in the target space, any embedding of $(\mathbb{R}^d, \|\cdot\|_2)$ into $(\mathbb{R}^k, \|\cdot\|)$ with distortion $\varepsilon$ generally requires $k = \Omega\big(\varepsilon^{-2} \log(\varepsilon^2 n)/\log(1/\varepsilon)\big)$, matching the optimal bound for the $\ell_2$ case up to a logarithmic factor.
\end{abstract}

	\section{Introduction}\label{sec 1}
	
	The Johnson-Lindenstrauss (JL) lemma is a cornerstone of dimensionality reduction, enabling the embedding of high-dimensional datasets into a lower-dimensional space while approximately preserving all of their pairwise Euclidean distances. Specifically, the JL lemma states that for any \(\varepsilon \in (0, 1)\) and finite set \(X \subseteq \mathbb{R}^d\) with \(n > 1\) elements, there exists a matrix \(\Phi \in \mathbb{R}^{k \times d}\) with \(k = \mathcal{O}(\varepsilon^{-2} \log n)\) such that
	\begin{equation}
	(1 - \varepsilon) \|{  x} - {  y}\|_2 \leq \|\Phi {  x} - \Phi {  y}\|_2 \leq (1 + \varepsilon) \|{  x} - {  y}\|_2 \quad \forall \, {  x}, {  y} \in X.
	\label{Equ:JLprop}
	\end{equation}
	Moreover, it has been shown that the dimension $k$ of the Euclidean space where $X$ is embedded above is optimal \cite{larsen2017optimality}. Remarkably, a matrix \(\Phi\) with independent Gaussian entries achieves this property with high probability. However, Gaussian matrices are computationally expensive to store and to multiply against extremely long vectors. This limitation has inspired extensive research into structured random matrices that admit fast matrix-vector multiplication algorithms while preserving the desirable JL-properties of Gaussian embeddings (see, e.g., \cite{Ailon2008fast,ailon2009the,krahmer2011new,bourgain2015toward,iwen2024on} among many others).
    
	The $\ell_p$ norms for $p \in [1,2)$ are particularly meaningful in applications such as nearest neighbor search, compressed sensing, and machine learning, where they offer varying degrees of robustness to outliers and noise, as well as the ability to promote sparsity in solutions \cite{Foucart2024,ailon2009the,Beyer1999when,Hinneburg2000what}. The $\ell_1$ norm provides maximum robustness, while $\ell_p$ norms with $p > 1$ offer a continuum between $\ell_1$ and $\ell_2$ properties.  A recent result by Dirksen, Mendelson, and Stollenwerk provides an optimal embedding from $\ell_2$ into $\ell_1$ using structured random circulant matrices, achieving both the optimal embedding dimension and a $\mathcal{O}(d \log d)$ computational complexity \cite{Dirksen2024fast}. In this work we complement their construction by showing that a simplified version of an embedding originally proposed by Ailon and Liberty in \cite{Ailon2008fast} for the $\ell_2$ setting can achieve a slightly improved $\mathcal{O}(d \log k)$ runtime for embeddings from $\ell_2$ into $\ell_p$ for any $p \in [1,2]$, provided that the embedding dimension satisfies $k\leq d^{1/4}$. Note that this improvement in the computational cost is most significant when $k$ is substantially smaller than $d$, making the limiting assumption that $k\leq d^{1/4}$ harmless in practice.\footnote{In fact, by utilizing Ailon and Liberty's original construction \cite{Ailon2008fast} one achieves the same $\ell_2$ into $\ell_1$ embedding result for all $k \lesssim d^{1/2}$.  However, the additional complication of their original construction only provides marginal value these days given that Dirksen et al.'s nice embedding result \cite{Dirksen2024fast} already achieves the same effective runtime complexity for all $k \geq d^{1/4}$.} Our main result is stated below.
    
	\begin{theorem}\label{theo: main}  
		Let \(\varepsilon, \rho \in (0, 1)\), $p \in [1,2]$, and \(X \subseteq \mathbb{R}^d\) be a finite set with \(|X| = n\). Suppose the embedding dimension \(k\) satisfies  
		\[
		k \geq \varepsilon^{-2}\max \left \{50C_0,216\log\left( \frac{6n^2}{\rho} \right) \right \} \quad \text{and} \quad k \leq d^{1/4},
		\]  
		where $C_0$ is the absolute constant in Lemma \ref{lem: berry-esseen}. Then, with probability at least \(1 - \rho\), the linear map \(\Psi \colon \mathbb{R}^d \to \mathbb{R}^k\), defined in \eqref{eq: def embedding}, can be applied to any input vector in \(\mathcal{O}(d \log k)\) time and satisfies  
		\[
		\left| \|\Psi({  x}) - \Psi({  y})\|_p - \|{  x} - {  y}\|_2 \right| \leq \varepsilon \|{  x} - {  y}\|_2 \quad \text{for all } {  x}, {  y} \in X.
		\]
	\end{theorem}

	The embedding $\Psi \colon \mathbb{R}^d \longrightarrow \mathbb{R}^k$ in Theorem~\ref{theo: main} is defined by
	\begin{equation}\label{eq: def embedding}
	\Psi {  x} = k^{-1/p} \beta_p^{-1} AD_1HD_2HD_3 {  x} \quad \forall \, {  x} \in \mathbb{R}^d,
	\end{equation}
	where \(A\) is a \(k \times d\) \(4\)-wise independent matrix, $H$ is the $d\times d$ Hadamard-Walsh matrix, and $D_1, D_2, D_3$ are independent copies of a $d\times d$ diagonal matrix with independent Rademacher entries (i.e., each diagonal entry is $\pm1$ with equal probability). The scaling factor $k^{-1/p}\beta_p^{-1}$, where $\beta_p^p = \mathbb{E}|Z|^p$ for $Z \sim \mathcal{N}(0,1)$, ensures proper normalization. This construction is a simplified version of an earlier embedding proposed by Ailon and Liberty \cite{Ailon2008fast} for the \(\ell_2\) norm, adapted here for the $\ell_p$ setting.
	
	Additionally, we establish a general lower bound on the target dimension $k$ for embeddings that preserve Euclidean distances when measured via an arbitrary norm in the target space.
	Let $\|\cdot\|$ be any norm on $\mathbb{R}^k$, $X \subseteq \mathbb{R}^d$ be any $n$-point subset, and suppose that $f \colon \mathbb{R}^d \to \mathbb{R}^k$ satisfies
\begin{equation}\label{eq: bilip}
		(1-\varepsilon)\|x-y\|_2 \leq \|f(x)-f(y)\| \leq (1+\varepsilon)\|x-y\|_2
	\end{equation}
	for all $x,y$ in $X$.  When $\|\cdot\|$ is the Euclidean norm, Larsen and Nelson \cite{dirksen2016dimensionality} proved the optimal bound $k = \Omega(\varepsilon^{-2} \log(\varepsilon^2 n))$ must generally hold for all  $\varepsilon \in (\log^{0.5001} n / \sqrt{\min\{n,d\}}, 1)$, which is almost the full meaningful range of $\varepsilon$. Their argument exploits the geometric structure induced by inner products in the Euclidean setting, however, making its extension to general norms challenging. Nonetheless, by adapting their argument we are able to obtain a lower bound for arbitrary norms that's both valid for the full range of $\varepsilon$, and optimal up to a logarithmic factor in $1/\varepsilon$.
	
	\begin{theorem}\label{theo: lower bound}
		Let $\|\cdot\|$ be any norm on $\mathbb{R}^k$. Suppose that for every $n$-point set $X \subseteq \mathbb{R}^d$ there exists a map $f_X \colon X \longrightarrow \mathbb{R}^k$
		satisfying
		\[
		(1-\varepsilon)\|x-y\|_2 \le \|f_X(x)-f_X(y)\| \le (1+\varepsilon)\|x-y\|_2 \quad \forall x,y \in X,
		\]
		with $\varepsilon \in \Big(\sqrt{\log n}/\sqrt{d}, 1\Big)$. 
		Then, the target dimension $k$ must satisfy
		\[
		k = \Omega\Big( \frac{\log(\varepsilon^2 n)}{\varepsilon^2 \log(1/\varepsilon)} \Big).
		\]
	\end{theorem}

	\subsection{Related Work}
	
	It is natural to ask whether JL-type embeddings extend to general $\ell_p$ norms. However, linear embeddings preserving pairwise $\ell_p$ distances up to $1 \pm \varepsilon$ are impossible for $p \neq 2$. Charikar and Sahai \cite{charikar2002dimension} (for $p=1$) and Lee, Mendel, and Naor \cite{james2005metric} (for general $p$) showed that any such linear map must incur distortion at least $\left( \frac{n}{k} \right)^{|1/p - 1/2|}$, implying that for $\ell_1$, the embedding dimension must be $k = \Omega(\varepsilon^{-2} n)$. This is much higher than the optimal $\mathcal{O}(\varepsilon^{-2} \log n)$ dimension guaranteed by the JL lemma for the Euclidean case.

	A more promising direction is to construct embeddings that map $\ell_2$ distances in the domain space to approximate $\ell_p$ distances in the target space. However, this setting introduces challenges not present in the $\ell_2$ case. Unlike the $\ell_2$ norm, where $\mathbb{E} \|\Psi x\|_2^2 = k$ uniformly for all unit vectors ${  x}$ when $\Psi$ has independent zero-mean, unit-variance entries, the expectation $\mathbb{E} \|\Psi x\|_p$ generally depends on the specific vector ${  x}$. Since embedding quality relies on concentration around the mean, this variability makes achieving low-distortion embeddings in $\ell_p$ more difficult.
	
	An important exception occurs when $\Psi$ has independent Gaussian entries. In this case, $\Psi x$ is a Gaussian vector, and the $p$-th moment satisfies $\mathbb{E} \|\Psi x\|_p^p = \beta_p^p \cdot k$ for all unit vectors $x$. While this makes Gaussian matrices theoretically ideal for $\ell_p$ embeddings, their unstructured nature leads to high computational costs, motivating the search for faster, structured alternatives.
    
	A significant breakthrough in this direction was the \textit{Fast Johnson-Lindenstrauss Transform (FJLT)}, introduced by Ailon and Chazelle in \cite{ailon2009the}. The FJLT can embed all \(\ell_2\)-distances between points in a set $X \subseteq \mathbb{R}^d$ of size $n$ into \(\ell_1\) with constant probability while simultaneously supporting $\mathcal{O}\left( d \log d + \min\{d \varepsilon^{-2} \log n, \varepsilon^{-3} \log^2 n\} \right)$-time matrix vector multiplications, thereby reducing the computational complexity of its predecessors \cite{Kushilevitz2000efficient}. This efficiency has made the FJLT an important tool in applications such as approximate nearest neighbor search. Further progress was made by Matou\v sek \cite{matousek2008on}, who showed that the Gaussian entries in the FJLT matrices can be replaced with \(\pm 1\) entries and still achieve $\ell_1$ embeddings, albeit with a slightly worse dependence on 
    parameters. 
	
	A major advancement was recently achieved in \cite{Dirksen2024fast} by Dirksen, Mendelson, and Stollenwerk, who constructed a fast embedding from $\ell_2$ into $\ell_1$ based on double circulant matrices. Their method significantly improves upon all previous constructions by simultaneously achieving optimal embedding dimension $k = \mathcal{O}(\varepsilon^{-2} \log |X|)$ and fast computation time $\mathcal{O}(d \log d)$. Moreover, the embedding is valid for sets $X \subseteq \mathbb{R}^d$ of size up to $|X| =\mathcal{O}(\exp(c \varepsilon^2 d / \log^6 d))$, which covers nearly all practical cases, since the maximum possible cardinality for distortion-preserving embeddings in $\mathbb{R}^d$ is only exponential in $d$. This result was obtained as a by-product of their main goal: constructing binary embeddings into the Hamming cube with optimal dimension and near-linear runtime. In their analysis they establish surprisingly strong concentration inequalities for double circulant matrices, despite the high dependence among their columns.
	
	\subsection{Our Contribution}
	
	We present a simple and fast linear embedding for finite sets that preserves Euclidean distances measured via the $\ell_p$ norm in the target space for any $p \in [1,2]$. Moreover, we provide explicit bounds and constants that quantify the relationship between the embedding dimension $k$, the distortion $\varepsilon$, and the probability of success. 
	
	The key technical ingredient in the proof of Theorem \ref{theo: main} is the following concentration result, which guarantees pointwise preservation in the $\ell_p$ norm with high probability.

    \begin{theorem}\label{theo: main 2}
		Let \({  x} \in \mathbb{R}^d\) with \(\| {  x}\|_2 = 1\) and \(\varepsilon \in (0, 1)\). Assume \(k \geq \varepsilon^{-2}\max\{4941, 50C_0\}\) and \(k \leq d^{1/4}\). Then, the embedding \(\Psi \colon \mathbb{R}^d \to \mathbb{R}^k\) defined in \eqref{eq:  def embedding} satisfies
		\[
		\prob{\left| \|\Psi {  x}\|_p - 1 \right| > \varepsilon} \leq 6 \exp\left( -\frac{k\varepsilon^2}{216} \right).
		\]
		Here, $C_0$ is the absolute constant from Lemma \ref{lem: berry-esseen}.
	\end{theorem}
	Our main result Theorem \ref{theo: main} follows from Theorem \ref{theo: main 2} via a standard union bound argument. Finally, Theorem \ref{theo: lower bound} shows that the dimension reduction provided by Theorem \ref{theo: main} is optimal up to logarithmic factors.

    \subsection{Organization of the Paper}
	
	The remainder of this paper is organized as follows. In Section \ref{sec 2}, we introduce the necessary preliminaries, definitions, and notation (including key concepts related to \(\ell_p\) norms, operator norms, and matrices) used throughout the paper.  Next, Section~\ref{sec 3} develops the key ingredients used in the proof of our main results. We begin by analyzing concentration of measure followed by the expected value of the embedding.  In Section \ref{sec 4}, we combine these results to prove Theorem \ref{theo: main 2}, and establish Theorem \ref{theo: main} as a corollary.  Finally, we prove Theorem \ref{theo: lower bound} in Section \ref{sec 5}.
    
\section{Preliminaries}\label{sec 2}

For convenience, we assume that both $d$ (the ambient dimension of the input space) and $k$ (the target dimension of the embedding) are powers of $2$. This assumption simplifies the arguments, particularly when working with structured matrices like $4$-wise independent matrices or the Walsh-Hadamard matrix defined below. This can be achieved with minimal impact on dimensionality reduction by increasing $d$ and $k$ to the nearest power of $2$ and padding the additional entries with zeros.

For any \(p \geq 1\), the \(\ell_p\) norm of a vector \({  x} \in \mathbb{R}^d\) is defined by \(\|{  x}\|_p = \left( \sum_{j=1}^d |x_j|^p \right)^{1/p}\). For a random variable $X$, we denote its $L_p$ norm by $\|X\|_p = (\mathbb{E}|X|^p)^{1/p}$.
The standard inner product between two vectors \({  x}, {  y} \in \mathbb{R}^d\) is denoted \(\langle {  x}, {  y} \rangle\), and is defined as $\langle {  x}, {  y} \rangle = \sum_{j=1}^d x_j y_j$.
A matrix is called \textit{signed} when all its entries are either \(+1\) or \(-1\). The \(i\)-th row of a matrix \(A\) is denoted \(A_{(i)}\), while the \(j\)-th column is denoted \(A^{(j)}\). For a matrix \(A \in \mathbb{R}^{k \times d}\), the \(\ell_p \to \ell_q\) \textit{operator norm} is defined as
\[
\|A\|_{p \to q} = \sup_{\|{  x}\|_p = 1} \|A {  x}\|_q.
\]

A function \(f \colon \mathbb{R}^d \to \mathbb{R}\) is called \textit{Lipschitz} if there exists a constant \(K > 0\) such that $|f({  x}) - f({  y})| \leq K \| {  x} - {  y}\|_2$ for all \({  x}, {  y} \in \mathbb{R}^d\).  The smallest such constant \(K\) is called the \textit{Lipschitz constant} of \(f\) and is denoted by \(\|f\|_L\).

A \textit{Rademacher diagonal matrix} is a diagonal matrix $D \in \mathbb{R}^{d \times d}$ whose diagonal entries are independent random variables $\xi_j$ taking values $+1$ or $-1$ with probability $1/2$. We denote the vector of diagonal entries by $\xi = (\xi_1, \ldots, \xi_d)$.

A matrix $A \in \mathbb{R}^{k \times d}$ is called \textit{4-wise independent} if for any $1 \leq i_1 < i_2 < i_3 < i_4 \leq k$ and any $(b_1, b_2, b_3, b_4) \in \{+1, -1\}^4$, the number of columns $A^{(j)}$ for which
\[
(A^{(j)}_{i_1}, A^{(j)}_{i_2}, A^{(j)}_{i_3}, A^{(j)}_{i_4}) = \frac{1}{\sqrt{k}} (b_1, b_2, b_3, b_4)
\]
is exactly $d / 2^4$. 

The \textit{Walsh-Hadamard matrix} $H_d \in \mathbb{R}^{d \times d}$ is a structured orthogonal matrix defined recursively for $d = 2^m$ as
\[
H_d = \frac{1}{\sqrt{2}} \begin{bmatrix}
	H_{d/2} & H_{d/2} \\
	H_{d/2} & -H_{d/2}
\end{bmatrix},
\]
with $H_1 = [1]$. It satisfies $H_d^T H_d = I_d$, where $I_d$ is the identity matrix. This makes $H_d$ an isometry, meaning it preserves the $\ell_2$ norm of any vector $x \in \mathbb{R}^d$. From now on, we will use $H$ to denote $H_d$ and drop the index $d$ for simplicity.

	\section{Technical Ingredients}\label{sec 3}
	
	This section develops the core technical tools required for our main results. Let \( k, d \in \mathbb{N} \) with \( k \leq d \) and $p \in [1,2]$. We consider random matrices \( \Phi : \mathbb{R}^d \to \mathbb{R}^k \) of the form $\Phi = k^{-1/p}AD$, where
	\begin{itemize}
		\item \( A \in \{+1, -1\}^{k \times d} \) is a \textit{signed matrix} (i.e., all entries are \( \pm 1 \)),
		\item $D$ is a \textit{Rademacher diagonal matrix}, that is, its diagonal entries $ \xi_1, \dots, \xi_d $ are independent random variables taking values \( \pm 1 \) with probability \( 1/2 \).
	\end{itemize}
	
	\subsection{Concentration of measure}

	We begin by providing conditions to guarantee that $\|\Phi x\|_p$ is highly concentrated around its expectation. We use the following powerful concentration inequality due to Talagrand, which corresponds to (1.6) in \cite{talagrand1991probability}.
	
	\begin{theorem}[Talagrand's concentration inequality]\label{theo: talagrand}
		Let $f\colon \mathbb{R}^d \longrightarrow \mathbb{R}$ be a Lipschitz convex map with Lipschitz constant $\|f\|_{L}$, and let $\xi \in \mathbb{R}^d$ be a Rademacher vector. Then
		\begin{equation}\label{eq: talagrand}
		\prob{ |f(\xi) - \mathbb{E}f(\xi)| > t} \leq 2 \exp\left( -\frac{t^2}{2 \|f\|_{L}^2} \right) \quad \forall \, t>0.
		\end{equation}
	\end{theorem} 
	To apply Theorem \ref{theo: talagrand}, we rewrite $\|\Phi x\|_p$ as a function of a Rademacher vector. Specifically, consider a $k \times d$ matrix $M$ defined column-wise as:
	\[
	M^{(j)} = A^{(j)} x_j \quad \text{for } j = 1, \dots, d.
	\]
	Then, observe that
	\[
	\|AD x\|^p_p = \sum_{i=1}^k \left| \sum_{j=1}^d A_{i,j} x_j \xi_j \right|^p = \sum_{i=1}^k |\langle M_{(i)}, \xi \rangle|^p = \|M \xi \|_p^p.
	\]
	In our case, the function $f( \xi) = \|k^{-1/p}M \xi\|_p$ is convex and Lipschitz with \begin{equation}\label{eq: lipschitz norm}
		\|f\|_L\leq \frac{1}{\sqrt{k}} \|M\|_{2\to 2}.
	\end{equation} 
	To see this, first note that the triangle inequality implies $\|f\|_L \leq k^{-1/p} \|M\|_{2 \to p}$. Moreover, for any $x \in \mathbb{S}^{d-1}$ H\"{o}lder's inequality yields
	\[\|Mx\|_p^p = \sum_{i=1}^k |\langle M_{(i)},x\rangle|^p \leq k^{1-p/2} \left (\sum_{i=1}^k |\langle M_{(i)},x\rangle|^2 \right )^{p/2} = k^{1-p/2} \|Mx\|_2^p,\]
	whence taking supremum over $x\in \mathbb{S}^{d-1}$ shows that $\|M\|_{2 \to p} \leq k^{1/p-1/2} \|M\|_{2 \to 2}$, proving the claim.
	
	The operator norm $\|M\|_{2 \to 2}$ has been studied extensively in the context of JL-type embeddings. Ailon and Liberty \cite{Ailon2008fast} showed that this quantity can be controlled when $A$ is a 4-wise independent matrix and the coordinates of $x$ are sufficiently well-spread (i.e., $\|x\|_4 = \mathcal{O}(d^{-1/4})$).
	\begin{lemma}[(5.2) and Lemma 5.1 in \cite{Ailon2008fast}]\label{lem: Ailon sigma}
		For the matrix $M$ defined by $M^{(j)} = A^{(j)} x_j$, we have
		\[
		\|M\|_{2 \to 2} \leq \|{  x}\|_4 \|A^T\|_{2 \to 4}.
		\]
		Additionally, if $A$ is a $k\times d$ $4$-wise independent matrix, then $\|A^T\|_{2 \to 4} \leq (3d)^{1/4}$.
	\end{lemma}
	
	The original statements of these results have been modified to be consistent with the scaling used in this work. These estimates, together with Talagrand's inequality (\ref{eq: talagrand}) and (\ref{eq: lipschitz norm}), imply a sub-gaussian concentration bound for $\|ADx\|_p$ around its expected value when $x$ is sufficiently flat.
	
	\begin{corollary}
		\label{cor: concentration simple}
		Let $A$ be a 4-wise independent $k \times d$ matrix, $D$ a Rademacher diagonal matrix, and $\Phi = k^{-1/p} AD$. Then for any $x \in \mathbb{R}^d$ with $\|x\|_2=1$ and $t\geq 0$ we have
		\[
		\prob{ \left| \|\Phi x\|_p - \mathbb{E}\|\Phi x\|_p \right| > t } \leq 2 \exp\left( \frac{-kt^2}{2 \|{  x}\|_4^2 (3d)^{1/2}} \right).
		\]
	\end{corollary}

    A common technique in dimensionality reduction is to precondition the input vector with a randomized isometry to ensure that its coordinates are well-spread or, in our particular case, to reduce its $\ell_4$ norm. Following an argument in \cite{Ailon2008fast}, we achieve this by applying a Hadamard transform composed with a diagonal Rademacher matrix.
	
	\begin{lemma}\label{lem: talagrand hadamard}
		Let ${  x} \in \mathbb{R}^d$ with $\|{ x}\|_2 = 1$, and let $D$ be a Rademacher diagonal matrix. Then, for any $t > 0$,
		\[
		\prob{\|HD {  x}\|_4 > (3^{1/4} + t)d^{-1/4} } \leq 2 \exp\left( -\frac{t^2}{2 \|{  x}\|_4^2} \right).
		\]
	\end{lemma}
	\begin{proof}
		The result follows from Talagrand's concentration inequality (Theorem~\ref{theo: talagrand}) and the well-known estimate $\mathbb{E}\|HDx\|_4 \leq 3^{1/4}d^{-1/4}$. Consider the convex Lipschitz function $f( \xi) = \|M'  \xi\|_4$, where $M^{\prime(j)} = H^{(j)} x_j$. To bound the Lipschitz constant, we note that
		\[
		\|M'\|_{2 \to 4} = \|M'^T\|_{4/3 \to 2} \leq \|{  x}\|_4 \cdot \|H^T\|_{4/3 \to 4}.
		\]
		By the Riesz-Thorin interpolation theorem, using that $\|H\|_{2 \to 2} = 1$ and $\|H\|_{1 \to \infty} = d^{-1/2}$, we obtain
		\[
		\|H^T\|_{4/3 \to 4} \leq d^{-1/4},
		\]
		which yields the desired sub-gaussian tail bound.
	\end{proof}
	
\begin{remark}
	A single application of Lemma \ref{lem: talagrand hadamard} only guarantees $\|HDx\|_4= \mathcal{O}(k^{1/2} d^{-1/4})$ with high probability, which is insufficient for our concentration bounds. Applying the transformation twice brings it down to $\mathcal{O}(d^{-1/4})$. This explains the use of the composition $HD_2HD_3$ in our embedding.
\end{remark}

\subsection{Expected value of the $\ell_p$ norm.}

The next step is to analyze the expected value of the projected $\ell_p$ norm $\|k^{-1/p}AD{  x}\|_p$ for $p \in [1,2]$. As above, let ${x} \in \mathbb{R}^d$ with $\|{  x}\|_2 = 1$, let $A$ be a $k \times d$ signed matrix (i.e., entries are $\pm 1$), and let $D$ be a $d \times d$ diagonal matrix whose diagonal entries are independent Rademacher random variables $\xi_j$ (i.e., $\xi_j = \pm 1$ with probability $1/2$).

Unlike the $\ell_2$ norm, the expectation of $\|k^{-1/p}AD{  x}\|_p$ generally depends on the vector ${  x}$. However, the next result shows that when the mass of ${x}$ is approximately uniformly distributed across its coordinates (i.e., when ${x}$ is ``flat''), this dependence becomes negligible. In such cases, we will show that the $p$-th moment of $\|k^{-1/p}AD{  x}\|_p$ is close to the $p$-th moment of a standard gaussian, with the deviation controlled by the $\ell_3$ norm of ${x}$.

\begin{lemma}\label{lem: expectation}
	Let $A$ be a $k\times d$ signed matrix, $D$ a $d\times d$ Rademacher diagonal matrix, and $x \in \mathbb{R}^d$ with $\|x\|_2=1$, and $\Phi=k^{-1/p}AD$. Then 
	\[ \left | (\ex \| \Phi x\|_p^p)^{1/p} - \beta_p\right | \leq (5C_0\| x\|_3^3)^{1/p},\]
	where $\beta_p^p=\mathbb{E}|Z|^p$ with $Z$ a standard Gaussian, and $C_0$ is the absolute constant from Lemma \ref{lem: berry-esseen}.
\end{lemma}

Lemma \ref{lem: expectation} follows from a quantitative version of the Central Limit Theorem: the random variable $Y_i = \sum_j A_{i,j} \xi_j x_j$, which corresponds to the $i$-th coordinate of $ADx$, approximates a standard gaussian when the coordinates of $x$ are well-spread. Hence, $\mathbb{E}|Y_i|^p \approx \mathbb{E}|Z|^p$. The following Berry-Esseen type result provides a quantitative measure of this approximation.

\begin{lemma}[Corollary 17.7 in \cite{Bhattacharya1976normal}]\label{lem: berry-esseen}
	Let $X_1, \dots, X_d$ be independent symmetric random variables with $\sum_{j=1}^d \mathbb{E}(X_j^2)= 1$. Let $F$ be the distribution function of $X_1 + \cdots + X_d$, and let $\Phi(t) = \frac{1}{\sqrt{2\pi}} \int_{-\infty}^t e^{-x^2/2} \, dx$ be the distribution function of the standard Gaussian. Then
	\[
	|F(t) - \Phi(t)| \leq \frac{C_0}{1 + |t|^3} \sum_{j=1}^d \mathbb{E}(|X_j|^3) \quad \forall \, t \in \mathbb{R},
	\]
	where $C_0$ is an absolute constant.
\end{lemma}

\begin{proof}[Proof of Lemma \ref{lem: expectation}]
	Observe that it suffices to prove that
	\begin{equation}\label{eq: p moment} 
		\left | \ex \| \Phi x\|_p^p - \beta_p^p\right | \leq 5C_0\|{  x}\|_3^3
	\end{equation}
	since the inequality 
\[
|a^{\alpha} - b^{\alpha}| \le |a - b|^{\alpha}, \qquad a,b \ge 0,\ \alpha \in (0,1],
\]
implies that
	\[ \left | (\mathbb{E}\|\Phi x\|_p^p)^{1/p} - \beta_p \right | \leq \left |\mathbb{E}\|\Phi x\|_p^p - \beta_p^p \right |^{1/p}.\]
    
	To prove (\ref{eq: p moment}), note that the random variables $Y_i = \sum_{j=1}^d A_{i,j} \xi_j x_j$, which represent the coordinates of $ADx$, are identically distributed. Therefore,  $\mathbb{E}\|k^{-1/p}AD{  x}\|^p_p=\mathbb{E}|Y_1|^p$.
    
	Define $X_j = A_{1,j} x_j \xi_j$ for $j = 1, \dots, d$, so that $Y_1 = \sum_{j=1}^d X_j$ is a sum of independent, symmetric random variables. Since $\sum_j \mathbb{E}(X_j^2) = 1$ and $\mathbb{E}(|X_j|^3) = |x_j|^3$, Lemma~\ref{lem: berry-esseen} gives 
	\begin{equation}\label{eq: berry essen bound} |F(t) - \Phi(t)| \leq \frac{C_0}{1 + |t|^3} \|x\|_3^3 \quad \forall \, t \in \mathbb{R}, \end{equation}
	where $F$ is the distribution function of $Y_1$, and $\Phi$ is the distribution function of a standard Gaussian. 
	
	Recall that if $Z$ follows a standard Gaussian distribution, we have defined $\beta_p^p=\mathbb{E}|Z|^p$. We can use (\ref{eq: berry essen bound}) to estimate the difference between $\beta_p^p$ and $\mathbb{E}|Y_1|^p$. Indeed, note that
	\[
	\left |\mathbb{E}|Y_1|^p - \beta_p^p\right | = \left| \int |t|^p \, dF(t) - \int |t|^p \, d\Phi(t) \right|.
	\]
	Integrating by parts, we have
	\begin{align*} \int |t|^p\, dF(t) &= \int \int_0^{|t|} ps^{p-1}ds \, dF(t) = \int_0^\infty ps^{p-1} \int_{|t|\geq s}  dF(t) \, ds \\
		&= \int_0^\infty ps^{p-1}(F(-s) + 1- F(s)) \, ds.
	\end{align*}
	Moreover, the same identity holds for $\Phi$. Therefore, we deduce that
	\begin{align*} \left | \int |t|^p d F(t) - \int |t|^p d\Phi(t)\right | &\leq  \int_{0}^\infty ps^{p-1}|F(-s) - \Phi(-s)| \,ds  +  \int_{0}^\infty ps^{p-1}  |F(s) - \Phi(s)| \,ds 
	\end{align*}
	Applying \eqref{eq: berry essen bound}, and recalling that $p \le 2$, we conclude
	\[
	\left |\mathbb{E}|Y_1|^p - \beta_p^p \right | \leq 2pC_0 \|{  x}\|_3^3 \int_{0}^\infty \frac{s^{p-1}}{1 + s^3} \, ds \leq \frac{8\pi C_0}{3\sqrt{3}} \|{x}\|_3^3 \leq 5C_0\|x\|_3^3,
	\]
	where the last inequality follows from bounding the integral.  
\end{proof}




	\section{Proof of main results}\label{sec 4}

    In this section, we prove Theorems~\ref{theo: main 2} and~\ref{theo: main}. We begin by formally defining the embedding~$\Psi$ used in both results.
	
	Let $X \subseteq \mathbb{R}^d$ be a finite set with $n$ elements, and let $\varepsilon \in (0,1)$. Take 
	\[k \geq \varepsilon^{-2}\max\{4941,50C_0\},\]
	where $C_0$ is the absolute constant appearing in Lemma \ref{lem: berry-esseen}, and assume that $k\leq d^{1/4}$. Let \(A\) be a \(k \times d\) \(4\)-wise independent matrix, let $D_1$, $D_2$, $D_3$ be independent Rademacher diagonal matrices, and let \(H\) be the \(d \times d\) Walsh-Hadamard matrix. The embedding \(\Psi \colon \mathbb{R}^d \to \mathbb{R}^k\) is defined by
	\begin{equation}\label{eq: embedding}
	\Psi {  x} = k^{-1/p}\beta_p^{-1} A D_1HD_2HD_3 {  x} \quad \forall \, {  x} \in \mathbb{R}^d.
	\end{equation}
	
	The construction is motivated as follows: 
	
	For vectors $x$ with well-spread coordinates, Corollary \ref{cor: concentration simple} ensures that \(A D_1\) preserves $\ell_p$ distances. The role of the isometry $HD_2HD_3$, is to guarantee this flatness condition by reducing the \(\ell_4\) norm of $x$ to \(\mathcal{O}(d^{-1/4})\) with high probability. The scaling factor \(k^{-1/p}\beta_p^{-1}\) normalizes the embedding so that \(\mathbb{E} \|\Psi {x}\|_p \approx 1\) for unit vectors.
	
	With this setup, we are finally ready to show that $\Psi$ preserves the Euclidean norms of vectors in \(\mathbb{R}^d\) via the \(\ell_p\) norm with high probability. 
    
	\begin{proof}[Proof of Theorem \ref{theo: main 2}]
		Let ${  x} \in \mathbb{R}^d$ with $\|{  x}\|_2=1$ and write $y=HD_3x$. Applying Lemma \ref{lem: talagrand hadamard} with $t=\sqrt{k}$ ensures that $\|{  y}\|_4 \leq 2k^{1/2} d^{-1/4}$ with probability at least $1-2e^{-k/2}$. Next, write $z=HD_2 y$ and apply Lemma \ref{lem: talagrand hadamard} again, this time with $t=3^{1/4}$, to obtain that $\|z\|_4 \leq 2\cdot 3^{1/4}d^{-1/4}$ with probability at least $1-4e^{-k/5}$, where we have used $k \leq d^{1/4}$.
        
		Write $Z=\|\Psi x\|_p$. By Lemma \ref{lem: expectation} we have
		\[ \left |\|Z\|_p - 1\right | \leq \beta_p^{-1}(5C_0\|z\|_3^3)^{1/p} \leq 5C_0\beta_p^{-1} \|z\|_4, \]
        where we have used that $\|z\|_3^3\leq \| z \|_2 \|z\|_4^2 = \|z\|_4^2$ by the Cauchy-Schwarz inequality, combined with the fact that $\|z\|_2=1$ since $H D_2 H D_3$ is unitary. Note that $\beta_p\geq \beta_1 = \sqrt{2/\pi}$ and recall that $d^{-1/4} \leq k^{-1} \leq \varepsilon^2/50C_0$.
		Then, we conclude that if $\|z\|_4 \leq 2\cdot 3^{1/4} d^{-1/4}$, which happens with probability at least $1-4e^{-k/5}$, then we have
		\begin{equation}\label{eq: final expectation}
			\left |\|Z\|_p - 1\right |\leq \frac{1}{3} \varepsilon^2 \leq \frac{1}{3} \varepsilon.
		\end{equation}
		As before, condition on the event $\|z\|_4 \leq 2\cdot 3^{1/4} d^{-1/4}$. Recall that Corollary \ref{cor: concentration simple} provides a tail bound for $Z - \mathbb{E}Z$, which implies it is sub-gaussian. It is standard that such tail bounds translate to moment bounds via the characterization of sub-gaussian norms. Specifically, Proposition 2.5.2 in \cite{vershynin2018high-dimensional} shows that if $\prob{|X| > t} \leq 2\exp(-t^2)$ for all $t\geq 0$, then $\|X\|_p \leq 3\sqrt{p}$ for all $p \geq 1$. Applying this to $X = (Z - \mathbb{E}Z)/K$ with $K^2 = \frac{2|z|_4^2 (3d)^{1/2}}{\beta_p k}$ yields
		\begin{equation}\label{eq: reverse triangle}
			|\mathbb{E}Z - \|Z\|_p| \leq \|Z-\mathbb{E}Z\|_p \leq 3\sqrt{p} \left (\frac{2\|z\|_4^2 3^{1/2}d^{1/2}}{\beta_p k}\right )^{1/2} \leq 3\sqrt{p} \left ( \frac{24}{\beta_p \sqrt{k}}\right)^{1/2}\leq \frac{3 \sqrt{61}}{\sqrt{k}} \leq \frac{\varepsilon}{3},
		\end{equation}
		for $k \geq 4941\varepsilon^{-2}$, where we have used $p \in [1,2]$ and $\beta_p \in [\sqrt{2/\pi},1]$. These estimations and the triangle inequality show that
		\[ |Z-1| \leq |Z-\mathbb{E}Z| + |\mathbb{E}Z - \|Z\|_p | + |\|Z\|_p - 1 | \leq |Z-\mathbb{E}Z| + \frac{2\varepsilon}{3}.\]
		Consequently, applying Corollary \ref{cor: concentration simple} and using $\beta_p\leq 1$ yields
		\[
			\prob{ |Z-1| >\varepsilon} \leq \prob{ |Z-\mathbb{E}Z|> \varepsilon/3} \leq 2 \exp\left( -\frac{k\varepsilon^2}{216} \right).
		\]
		Removing the conditioning, we conclude that
		\begin{equation}\label{eq: final bound}\prob{ \left| \|\Psi x\|_p - 1\right| >\varepsilon}\leq 2 \exp\left( -\frac{k\varepsilon^2}{216}\right ) + 4\exp\left (-\frac{k}{5}\right ) \leq 6\exp\left( -\frac{k\varepsilon^2}{216}\right ).\qedhere
		\end{equation}
		
		\end{proof}

		\begin{proof}[Proof of Theorem \ref{theo: main}]
			Without loss of generality we can assume $n\geq 3$. Thus, $k$ also satisfies the assumptions in Theorem \ref{theo: main 2}. Let \(T = \left\{ \frac{{  x} - {  y}}{\|{  x} - {  y}\|_2} : {  x}, {  y} \in X, {  x} \neq {  y} \right\}\). Since \(X\) has \(n\) elements, the set \(T\) has cardinality \(|T| \leq n^2\). By Theorem \ref{theo: main 2}, for any \({  x} \in T\), we have
			\[
			\prob{ \left| \|\Psi x\|_p - 1 \right| > \varepsilon } \leq 6 \exp\left( -\frac{k\varepsilon^2}{216} \right).
			\]
			Applying the union bound over all \({  x} \in T\), the probability that there exists \({  x} \in T\) with \(\left| \|\Psi x\|_p - 1 \right| > \varepsilon\) is at most
			\[
			|T| \cdot 6 \exp\left( -\frac{k\varepsilon^2}{216} \right) \leq 6n^2\exp\left( -\frac{k\varepsilon^2}{216} \right) \leq \rho
			\]
			as long as
			\[
			k \geq 216 \frac{\log(6\rho^{-1}n^2)}{\varepsilon^2}.
			\]
			
			Finally, \(\Psi\) is constructed using the same framework as in \cite{Ailon2008fast}, which involves fast matrix-vector multiplications with Walsh-Hadamard matrices, sparse Rademacher diagonal matrices, and a \(4\)-wise independent matrix. In particular, $\Phi x$ can be computed in \(\mathcal{O}(d \log k)\) time using the partial Hadamard transform and fast \(4\)-wise independent matrix multiplication techniques described in \cite{Ailon2008fast}.
		\end{proof}

		\section{Optimality of General Norm Embeddings}\label{sec 5}
		
		In this section, we establish a lower bound on the target dimension $k$ for embeddings that preserve Euclidean distances when measured via an arbitrary norm $\|\cdot\|$ in the target space. Given $\varepsilon \in (0,1)$, assume that for any set $X \subseteq \mathbb{R}^d$ with $n$ elements there is a map $f_X \colon \mathbb{R}^d \to \mathbb{R}^k$ that satisfies
		\begin{equation}\label{eq: bilip}
			(1-\varepsilon)\|x-y\|_2 \leq \|f_X(x)-f_X(y)\| \leq (1+\varepsilon)\|x-y\|_2 \quad \forall \, x,y \in X.
		\end{equation}
		Theorem \ref{theo: lower bound} states that the target dimension must satisfy $k=\Omega\left (\varepsilon^{-2} \log (\varepsilon^2 n)/\log(1/\varepsilon)\right )$. To achieve this lower bound, we adapt a combinatorial encoding technique of Larsen and Nelson~\cite{larsen2017optimality}, 
		which yields the optimal lower bound in the Euclidean case. 
		
		At a high level, we construct a large collection $\mathcal{P} = \{P_1, P_2, \ldots\}$ of $n$-point subsets of $\mathbb{R}^d$ that exhibit distinct geometric structures. Each geometry must be approximately preserved by its corresponding embedding \(f_{P_i}\) satisfying~\eqref{eq: bilip}. If the target dimension \(k\) is too small, the space \(\mathbb{R}^k\) cannot accommodate such a large number of geometrically distinct configurations, leading to a counting contradiction.
		
		Let $e_1,\ldots,e_d$ be the canonical basis of $\mathbb{R}^d$. Fix $s=\lfloor 1/(128\varepsilon^2)\rfloor$ and, for any $S\subseteq\{1,\ldots,d\}$ with $|S|=s$, define
		\[
		y_S = \frac{1}{\sqrt{s}}\sum_{j\in S} e_j.
		\]
		For any choice of $d$ sets $S_1,\ldots,S_d \subseteq \{1,\ldots,d\}$ of $s$ indices each, we consider a set 
		\begin{equation}\label{eq: P}
			P=\{0,e_1,\ldots,e_d,y_{S_1},\ldots,y_{S_d}\},
		\end{equation}
		and let $\mathcal{P}$ be the collection of all such sets.
		The key observation is that the Euclidean distances between $e_j$ and $y_S$ encode the indicator of $j\in S$, up to a gap of order $\varepsilon$.  
		If $f$ is a map satisfying~\eqref{eq: bilip}, these gaps are preserved when distances are measured in $\|\cdot\|$, allowing one to recover each subset $S_k$ from the image $f(P)$. By discretizing the image space via an $\varepsilon$-cover of the unit ball of $\|\cdot\|$, one obtains an injective encoding from $\mathcal{P}$ to $\{0,1\}^L$, where $L$ is the number of bits required to encode the $\varepsilon$-cover. Comparing the cardinalities of both sides yields the claimed lower bound.
		
		\begin{remark}
			Notice that for our construction to make sense we need $s \ge 1$, and thus 
			\[
			\varepsilon \le \frac{1}{\sqrt{128}}.
			\]
			This does not pose a problem for our lower bound. Indeed, for $\varepsilon > 1/\sqrt{128}$, Theorem~\ref{theo: lower bound} would aim to give $k = \Omega(\log n)$, which is straightforward. For example, if $d=n$, the image of the canonical vectors form a well-separated subset of $2B$, and standard packing arguments yield the bound. The proof of Theorem~\ref{theo: lower bound} also shows how the case $d \neq n$ can be reduced to this one. Thus, we may assume $\varepsilon \le 1/\sqrt{128}$ without loss of generality.
		\end{remark}

		Let $B$ denote the unit ball of $\|\cdot\|$ in $\mathbb{R}^k$.  Fix $N$ to be an $\varepsilon$-cover of $2B$ with respect to the norm $\|\cdot\|$, constructed using translates of $\varepsilon B$.  Let $P\in \mathcal{P}$ and $f \colon P \longrightarrow \mathbb{R}^k$ satisfy (\ref{eq: bilip}). Up to shifts, we may assume that $f(0)=0$.  Observe then that $f(P) \subseteq (1+\varepsilon)B\subseteq 2B$.  For each canonical vector $e_j$, choose $z_j \in N$ to be a representative of $f(e_j)$, i.e., $\|z_j - f(e_j)\| \le \varepsilon$. Likewise, for each $y_S \in P$, choose $z_S \in N$ satisfying $\|z_S - f(y_S)\| \le \varepsilon$. The following lemma shows that the collection of representatives $\{z_1, \dots, z_d, z_{S_1}, \dots, z_{S_d}\}$ uniquely determines the set $P$, as it encodes which indices $j$ belong to each subset $S$.

		\begin{lemma} \label{lem:gap_preserve_cover}
		With the notation above, for any $y_S \in P$, there exists a threshold $\tau \in \mathbb{R}$ such that for every $j \in \{1,\ldots,d\}$,
		\[
		j \in S \quad \mbox{if, and only if,} \quad \|z_j - z_S\| \le \tau.
		\]
		\end{lemma}

		\begin{proof}
			Let $d_{\rm in} = \|e_j - y_S\|_2$ for $j \in S$ and $d_{\rm out} = \|e_j - y_S\|_2$ for $j \notin S$. A direct computation gives
			\[
			d_{\rm in} = \sqrt{2 - \frac{2}{\sqrt{s}}}, \quad d_{\rm out} = \sqrt{2}.
			\]
			By the JL condition~\eqref{eq: bilip}, we have
			\[
			\|f(e_j) - f(y_S)\| \in [(1-\varepsilon)d_{\rm in}, (1+\varepsilon)d_{\rm in}] \quad \text{if } j \in S,
			\]
			and
			\[
			\|f(e_j) - f(y_S)\| \in [(1-\varepsilon)d_{\rm out}, (1+\varepsilon)d_{\rm out}] \quad \text{if } j \notin S.
			\]
			Triangle inequality with the representatives $\|z_j - f(e_j)\| \le \varepsilon$ and $\|z_S - f(y_S)\| \le \varepsilon$ gives
			\[
			\|z_j - z_S\| \in [(1-\varepsilon)d_{\rm in} - 2\varepsilon, (1+\varepsilon)d_{\rm in} + 2\varepsilon] \quad \text{if } j \in S,
			\]
			and
			\[
			\|z_j - z_S\| \in [(1-\varepsilon)d_{\rm out} - 2\varepsilon, (1+\varepsilon)d_{\rm out} + 2\varepsilon] \quad \text{if } j \notin S.
			\]
			Our choice of $s$ guarantees that these intervals are disjoint, and thus any threshold $\tau$ in between separates the two cases. To show this, observe that by the mean-value theorem we have $d_{\rm out} - d_{\rm in}\geq 1/\sqrt{2s}$. Hence, taking $s\leq 1/(128\varepsilon^2)$ gives $d_{\rm out} - d_{\rm in}\geq 8\varepsilon$ and the rest is routine.
		\end{proof}

		For our $\varepsilon$-cover $N$ of $2B$ define a bijection
		\[
		b: N \longrightarrow \{0,1\}^{\lceil \log_2 |N| \rceil},
		\]
		so that $b(z)$ gives a binary representative of each $z \in N$.  For a set $P \in \mathcal{P}$, recall that $z_j \in N$ is a representative for $f_P(e_j)$ and $z_S \in N$ is a representative for $f_P(y_S)$, where $f_P$ satisfies \eqref{eq: bilip} with $f_P(0) = 0$. We define the encoding of $P$ by concatenating the bit strings of its representatives:
		\[
		\text{enc}(P) = \big(b(z_1), \dots, b(z_d), b(z_{S_1}), \dots, b(z_{S_d})\big) \in \{0,1\}^L,
		\]
		where $L = 2d \lceil \log_2 |N| \rceil$. From the previous lemma, we infer that two distinct sets $P$, $P' \in \mathcal{P}$ cannot be assigned to the same encoding.
		
		\begin{lemma} \label{lem:encoding_injective}
			The encoding map $\text{enc} \colon \mathcal{P} \longrightarrow \{0,1\}^L$ defined above is injective.
		\end{lemma}
	
		\begin{proof}
			Suppose $\text{enc}(P) = \text{enc}(P')$ for two sets $P, P' \in \mathcal{P}$.  
			By definition of $\text{enc}$ and the bijection $b$, this implies that the corresponding representatives coincide:
			\[
			z_j = z'_j \quad \forall j \in \{1,\ldots,d\} \quad \mbox{and} \quad z_{S_k} = z'_{S_k} \quad \forall \, k \in \{1,\ldots,d\}.
			\]
			By Lemma~\ref{lem:gap_preserve_cover}, the distances between the $z_j$ and $z_{S_k}$ allow us to uniquely recover each subset $S_k$ in $P$ from the representatives.  
			Since the subsets in $P$ and $P'$ are recovered identically from their representatives, we must have $S_k = S'_k$ for all $k$, and hence $P = P'$.
		\end{proof}
		
		To establish Theorem \ref{theo: lower bound}, it remains to compare the cardinalities of $\mathcal{P}$ and $\{0,1\}^L$.
		
		\begin{proof}[Proof of Theorem~\ref{theo: lower bound}]
			We first study the critical case where the dimension satisfies 
			\[
			d = \frac{n-1}{2}.
			\]
			Under this assumption, the cardinality of any set $P \in \mathcal{P}$ is precisely $n$. Let $N$ be a minimal $\varepsilon$-cover of $2B$ with respect to $\|\cdot\|$. Thus, $|N| \le (6/\varepsilon)^k$ by standard volumetric arguments.  
			Then $\lceil \log_2 |N| \rceil = \mathcal{O}(k \log(1/\varepsilon))$, and the encoding length of each $P \in \mathcal{P}$ satisfies
			\[
			L = 2d \, \lceil \log_2 |N| \rceil = \mathcal{O}(d k \log(1/\varepsilon)).
			\]
			By Lemma~\ref{lem:encoding_injective}, the map $\text{enc} \colon \mathcal{P} \to \{0,1\}^L$ is injective, and thus $|\mathcal{P}| \le 2^L$. Notice that by the construction of $\mathcal{P}$, we have
			\[
			|\mathcal{P}| = \binom{d}{s}^d \ge \left(\frac{d}{s}\right)^{sd} \geq \left(C d \varepsilon^2 \right)^{cd/\varepsilon^2}
			\]
			for some absolute constants $c$, $C>0$. Comparing these two bounds gives
			\[
			2^{\mathcal{O}(d k \log(1/\varepsilon))} \geq \left(C d \varepsilon^2 \right)^{cd/\varepsilon^2}.
			\]
			Taking logarithms and solving for $k$ yields
			\[
			k = \Omega\Big( \varepsilon^{-2} \log (\varepsilon^2 n) / \log(1/\varepsilon) \Big).
			\]
			The remaining cases follow from the critical case $d=(n-1)/2$. 
            
			If $d > (n-1)/2$, we simply restrict our construction \eqref{eq: P} to the first $(n-1)/2$ coordinates, ignoring the rest so that $|P| = n$ remains true for all $P \in \mathcal{P}$.
            			
			If $d < (n-1)/2$, let $P \subseteq\mathbb{R}^{(n-1)/2}$ be a set of the form (\ref{eq: P}). By the Johnson-Lindenstrauss lemma, if $\varepsilon \geq \sqrt{\log n/d}$, there exists a map $h \colon \mathbb{R}^{(n-1)/2} \longrightarrow \mathbb{R}^d$ that approximately preserves all pairwise distances of $P$ up to distortion $\mathcal{O}(\varepsilon)$. Composing this map with a hypothetical embedding $f: \mathbb{R}^d \to \mathbb{R}^k$ satisfying the desired distortion condition (\ref{eq: bilip}) with $X = P$ gives
			\[
			f \circ h : \mathbb{R}^{(n-1)/2} \to \mathbb{R}^k,
			\]
			which is an embedding of $P$ with distortion $\mathcal{O}(\varepsilon)$. Up to constant factors, this imposes on $k$ the same lower bound as in the critical case $d=(n-1)/2$.
		\end{proof}

\bibliographystyle{plain}
\bibliography{bibliography}

\end{document}